\documentclass[a4paper,10pt,oneside]{article}

\usepackage{amsmath,amssymb,amsthm,bm,booktabs,caption,float,graphicx,titlesec}
\usepackage[titletoc, title]{appendix}
\usepackage[pdfpagelabels, hyperindex, colorlinks, pagebackref]{hyperref}

\newtheorem{coro}{Corollary}[section]

\newtheorem{lem}{Lemma}[section]
\newtheorem{thm}{Theorem}[section]

\newcommand{\ddiv}{\operatorname{div}}

\newcommand{\verti}[1]{\left|#1\right|}
\newcommand{\vertiii}[1]{
  \left\vert\kern-0.25ex
  \left\vert\kern-0.25ex
  \left\vert #1
  \right\vert\kern-0.25ex
  \right\vert\kern-0.25ex
  \right\vert
}
\newcommand{\norm}[1]{\left\lVert#1\right\rVert}
\newcommand{\bndint}[1]{\sum_{T\in\mathcal T_h}\langle #1 \rangle_{\partial T}}
\newcommand{\bint}[2]{\left\langle #1 \right\rangle_{#2}}
\newcommand{\dual}[1]{\left\langle #1 \right\rangle}

\numberwithin{equation}{section}
\begin{document}
\title
{
  \Large\bf Analysis of a family of HDG methods for second order elliptic problems
  \thanks{
    This work was supported in part by National Natural Science Foundation of
    China (11171239) and Major Research Plan of National Natural Science
    Foundation of China (91430105).
  }
}

\author
{
  Binjie Li\thanks{Email: libinjiefem@yahoo.com}, \quad
  Xiaoping Xie \thanks{Corresponding author. Email: xpxie@scu.edu.cn} \\
  {School of Mathematics, Sichuan University, Chengdu 610064, China}
}

\date{}
\maketitle
\begin{abstract}
  In this paper, we analyze a family of hybridizable discontinuous Galerkin (HDG) methods
  for second order elliptic problems in two and three dimensions. The methods use piecewise
  polynomials of degree $k\geqslant 0$ for both the flux and numerical trace, and
  piecewise polynomials of degree $k+1 $ for the potential. We establish error estimates for the numerical flux and potential under the minimal
  regularity condition.
  Moreover, we construct a local postprocessing for the flux, which produces a
  numerical flux with better conservation. Numerical experiments in two-space dimensions
  confirm our theoretical results.
\end{abstract}

\section{Introduction}
The pioneering works on hybrid (also called mixed-hybrid) finite element methods
are due to Pian \cite{Pian1964} and Fraejis de Veubeke \cite{Veubeke1965} for the
numerical solution of linear elasticity problems. Here the term ``hybrid", as stated
in \cite{Atluri-Murakawa1977}, means ``the constraints on displacement
continuity and/or traction reciprocity at the inter-element boundaries are relaxed
a priori" in the hybrid finite element model. One may refer to
\cite{Pian-Tong1969, Pian1972,pian1984rational,Punch-Atluri1984,Pian-Tong1986,Pian-Wu,
Pian1995, Sze1992, Sze1993, xie2004optimization,Zhang-Xie2010} 
and to \cite{Simo-Rifai1990,Reddy-Simo1995,Kasper-Taylor2000}  respectively for
some developments of hybrid stress (also called assumed stress) methods and hybrid
strain (also called enhanced assumed strain)    methods based on generalized
variational principles, such as Hellinger-Reissner principle and Hu-Washizu
principle. In \cite{Braess1998,Zhou-Xie2002, Braess-C-Reddy2004,yu2011uniform},
stability and convergence were analyzed for several 4-node hybrid stress/strain
quadrilateral/rectangular elements. We refer to \cite{Brezzi1974,Brezzi-Marini1975,
Brezzi-Fortin1991} for the analysis of hybrid methods for 4th order elliptic
problems, and to \cite{Babuska-Oden-Lee1977, Oden-Lee1977,Raviart-Thomas1977,Raviart-Thomas1979}
for the analysis for second-order elliptic boundary-value problems. One may see
\cite{Ciarlet1978, Brezzi-Fortin1991, Roberts-Thomas1991, Pian1995} for more
references therein on the hybrid methods.

Due to the relaxation of function continuity at the inter-element boundaries,
the hybrid finite element model allows for piecewise-independent approximation
to the displacement/potential or stress/flux solution, thus leading to a sparse,
symmetric and positive definite discrete system through local elimination of
unknowns defined in the interior of the elements. This is one main advantage
of the hybrid methods. The process of local elimination is also called ``static
condensation" in engineering literature. In the discrete system, the unknowns
are only the globally coupled degrees of freedom of the approximation trace of
the `` displacement" or ``traction"  defined only on the boundaries of the elements.

In \cite{unified_hyd} Cockburn et.\ al.\ introduced a unifying framework for hybridization
of finite element methods for the second order elliptic problem: find the potential $u$
and the flux $ \bm\sigma$ such that
\begin{equation}\label{eq:model}
  \left\{
  \begin{array}{rl}
    \bm c\bm\sigma - \bm\nabla u = 0 & \text{in $ \Omega$},\\
    -\ddiv\bm\sigma = f & \text{in $ \Omega$},\\
    u = g & \text{on $ \partial\Omega$},
  \end{array}
  \right.
\end{equation}
where $ \Omega\subset\mathbb R^d$ is a polyhedral domain,
$ \bm c(x)\in [L^\infty(\Omega)]^{d\times d}$ is a
matrix valued function that is symmetric and uniformly positive definite
on $ \Omega$, $f\in L^2(\Omega) $ and $g\in H^{\frac{1}{2}}(\partial\Omega) $.
Here hybridization denotes the process to rewrite a finite element method
in a hybrid version. The unifying framework includes as particular cases
hybridized versions of mixed methods \cite{ArnoldBrezzi1985, BDM;1985,
COCKBURN-GOPALAKRISHNAN2004}, the continuous Galerkin (CG) method
\cite{Cockburn-Gopalakrishnan-Wang2007}, and a wide class of hybridizable discontinuous
Galerkin (HDG) methods. In \cite{unified_hyd}  three new kinds of HDG methods,
or more precisely LDG-H (Local DG-hybridizable) methods, were presented, where
the unknowns are the approximations of the potential $u$ and flux $ \bm\sigma$,
defined in the interior of elements, and the numerical trace of $u$, defined
on the interface of the elements. In \cite{projection_based_hdg} an error
analysis was carried out for one of the three HDG methods by \cite{unified_hyd},
based on the use of a projection operator inspired by the form of the numerical
traces of the methods. Following the same idea as in \cite{projection_based_hdg},
a unifying framework was proposed in \cite{cond_super_con_hdg} to analyze a large
class of methods including the hybridized versions of some mixed methods as well
as several HDG methods. We note that in \cite{Lehrenfeld2010}, a reduced HDG scheme
was proposed, which only includes the potential approximation and numerical trace
as unknowns. Recently this HDG method was analyzed in \cite{Oikawa2014} for the
Poisson problem.

In this paper, we analyze a family of HDG methods for problem \eqref{eq:model}.
We use piecewise polynomials of degree $k$ for both the numerical flux
$ \bm\sigma_h $, and the numerical trace $ \lambda_h $ of $u$, and use piecewise
polynomials of degree $k+1 $ for the numerical potential $u_h $. It should be
mentioned that, in \cite{Qiu, Qiu;2}, the same methods have been analyzed for
convection diffusion equations with constant   diffusion coefficient and linear
elasticity problems, respectively. We note that   in our analysis the diffusion
coefficient $ \bm c(x)\in [L^\infty(\Omega)]^{d\times d}$ is a matrix valued
function. By following a similar idea of \cite{Gudi;2010}, we
establish   error estimates for the numerical flux and potential under the minimal
regularity condition, i.e.
\[
  u\in H^1(\Omega)\text{ and }\bm\sigma\in H(\ddiv;\Omega).
\]
This is significant since the regularity  $u\in H^{1+\alpha}(\Omega) $ may not
hold for $ \alpha>0.5$ for practical problems. To our best knowledge, such a
kind of error estimation has not been established for HDG methods, yet.  We note
that in \cite{Cockburn;2014} an error estimate was established under the
condition that $u\in H^{1+\alpha}(\Omega) $($ \alpha>0.5$),  while the estimate
with $ \alpha<0.5$ is fundamental to the multi-grid method developed there.  In
our contribution, we also construct a local postprocessing for the flux, which produces
a numerical flux $ \bm\sigma_h^*$ with better conservation.

The rest of this paper is organized as follows. In Section \ref{HDG-method} we
follow the general framework in \cite{unified_hyd} to describe the corresponding
HDG methods. Section \ref{sec_conv} is devoted to the error estimation for the
numerical flux and potential under the minimal regularity condition.
%
Section \ref{sec_post} presents the postprocessing for the flux. Finally Section
\ref{sec_num} provides numerical results.

\section{HDG method}\label{HDG-method}
Let us start by introducing some geometric notations. Let $ \mathcal T_h $ be a
conventional conforming and shape-regular triangulation of $ \Omega$, and let
$ \mathcal F_h $ be the set of all faces of $ \mathcal T_h $. For any $T\in\mathcal
T_h $, we denote by $h_T$ the diameter of $T$ and set $h := \max_{T\in\mathcal
T_h}h_T$. For any $T\in\mathcal T_h $ and $F\in \mathcal F_h $, let $V(T) $, $M(F) $
and $ \bm W(T) $ be local spaces of finite dimensions. Then we define
\begin{align}
  V_h &:= \left\{
    v_h\in L^2(\Omega): v_h|_T \in V(T) \quad
    \text{for all $T\in\mathcal T_h $}
  \right\}, \\
  M_h &:= \left\{
    \mu_h\in L^2( \mathcal F_h ): \mu_h|_{F} \in M(F) \quad
    \text{for all $F\in\mathcal F_h $}
  \right\}, \\
  \bm W_h &:= \left\{
    \bm\tau_h\in [L^2(\Omega)]^d: \bm\tau_h|_T \in \bm W(T) \quad
    \text{for all $T\in\mathcal T_h $}
  \right\}.
\end{align}
For any $g\in L^2(\partial\Omega) $, set
\[
  M_h(g): = \{\mu_h\in M_h: \bint{\mu_h,\eta_h}{\partial\Omega} = \bint{g, \eta_h}{\partial\Omega}
  \text{ for all $ \eta_h\in M_h $}\},
\]
and define $M_h^0:=M_h(0) $. In addition, for any $T\in\mathcal T_h $, define
\begin{equation}
  M(\partial T):=\left\{\mu \in L^2(\partial T):
  \mu |_F \in M(F)\text{ for all face $F$ of $T$}\right\},
\end{equation}
and then define $ \mathcal P_T^\partial : H^1(T) \to M(\partial T) $ by
\begin{equation}\label{projection1}
  \dual{\mathcal P_T^\partial v, \mu}_{\partial T} = \dual{v, \mu}_{\partial T}
  \quad\text{for all $v\in H^1(T) $ and $ \mu \in M(\partial T) $.}
\end{equation}
Above and in what follows, for any polyhedral domain $D\subset\mathbb R^d$, we
use $(\cdot,\cdot)_D$ and $ \bint{\cdot,\cdot}{\partial D}$ to denote the
$L^2$-inner products in $L^2(D) $ and $L^2(\partial D) $ respectively, and for
convenience, we shall use $(\cdot,\cdot) $ to abbreviate $(\cdot,\cdot)_\Omega$.

Following \cite{unified_hyd}, the general framework of HDG methods is as
follows: seek $(u_h,\lambda_h,\bm\sigma_h) \in V_h \times M_h(g) \times \bm
W_h $, such that
\begin{subequations}\label{discretization}
  \begin{align}
    (\bm c\bm\sigma_h,\bm\tau_h) + (u_h,\ddiv_h\bm\tau_h) -
    \sum_{T \in \mathcal T_h} \dual{
      \lambda_h,\bm\tau_h\cdot\bm n
    }_{\partial T} &= 0,
    \label{eq:hdg-1} \\
    -(v_h,\ddiv_h\bm\sigma_h) +
    \sum_{T \in \mathcal T_h} \dual{
      \alpha_T(\mathcal P_T^\partial u_h - \lambda_h), v_h
    }_{\partial T} &= (f, v_h),
    \label{eq:hdg-2} \\
    \sum_{T \in \mathcal T_h} \dual{
      \bm\sigma_h\cdot\bm n - \alpha_T(\mathcal P_T^\partial u_h-\lambda_h),
      \mu_h
    }_{\partial T} &= 0
    \label{eq:hdg-3}
  \end{align}
\end{subequations}
hold for all $(v_h,\mu_h,\bm\tau_h)\in V_h\times M_h^0 \times \bm W_h $, where
the broken divergence operator, $ \ddiv_h $, is given by
\[
  (\ddiv_h\bm\tau_h)|_T: = \ddiv(\bm\tau_h|_T)
  \quad\text{for all }T\in\mathcal T_h,\bm\tau_h \in
\bm W_h,
\]
and $ \alpha_T$ denotes a nonnegative penalty function defined on $ \partial T$.

In this paper we choose the local spaces $V(T),\ M(F), \bm W(T) $ and the penalty parameter $ \alpha_T$
as follows:
\begin{equation}\label{ours1}
  V(T) = P_{k+1}(T),\quad
  M(F)  = P_k(F),\quad
  \bm W(T) = [P_k(T)]^d,
\end{equation}
\begin{equation}\label{ours2}
  \alpha_T  = h_T^{-1},
\end{equation}
where, for any nonnegative integer $j$, $P_j(T) $ and $P_j(F) $ denote the sets of
polynomials with degree $ \leqslant j$ on $T$ and $F$ respectively. It is easy to
obtain the following existence and uniqueness results.
\begin{lem} The HDG scheme \eqref{discretization} with the choices \eqref{ours1}-\eqref{ours2} admits
  a unique solution $(u_h,\lambda_h,\bm\sigma_h)\in V_h
  \times M_h(g) \times \bm W_h $.
\end{lem}
\section{Error analysis}\label{sec_conv}
We first introduce some notation and conventions. In the rest of this paper, we
shall use the standard definitions of Sobolev spaces and their (semi-)norms
\cite{ADAMS}, namely, for an arbitrary open set $D\subset\mathbb R^d$ and any
positive integer $s$,
\begin{align*}
  H^s(D) &:= \{v\in L^2(D): \partial^{\alpha}v \in L^2(D)\text{ for all $|\alpha| \leqslant s$}\},\\
  \norm{v}_{s,D} &:= \left(
    \sum_{|\alpha|\leqslant s}\int_{D}|\partial^{\alpha}v|^2
  \right)^\frac12, \quad
  \verti{v}_{s,D} := \left(
    \sum_{|\alpha|= s}\int_{D}|\partial^{\alpha}v|^2
  \right)^\frac12 \quad\text{for all $v\in H^s(D) $.}
\end{align*}
We use $ \norm{\cdot}_D$ and $ \norm{\cdot}_{\partial D}$ to denote the
$L^2$-norms in $L^2(D) $ and $L^2(\partial D) $ respectively; in particular, we
shall use $ \norm{\cdot}$ to abbreviate $ \norm{\cdot}_\Omega$. For any
$T\in\mathcal T_h $ and nonnegative integer $j$ , let $ \mathcal P_T^j:L^2(T)\to
P_j(T) $ be the standard $L^2$-orthogonal projection operator, and define the operator $
\mathcal P_h^j$ by
\[
  (\mathcal P_h^j v)|_T := \mathcal P_T^j v \quad
  \text{for all $T \in \mathcal T_h $ and $ v\in L^2(\Omega) $.}
\]

In the rest of this paper, $x \lesssim y$ (or $x \gtrsim y$) denotes that there
exists a positive constant $C$ such that $x \leqslant Cy$ (or $x\geqslant Cy$),
where $C$ only depends on $ \bm c$, $k$, $ \Omega$, or the regularity of $ \mathcal
T_h $. The notation $x \sim y$ abbreviates $x \lesssim y\lesssim x$.

\subsection{Auxiliary interpolation operators}
Define
\[
  V_h^\text{c} := \left\{
    v_h^c\in H^1_0(\Omega):
    v_h^c|_T\in P_1(T) \text{ for all $T\in\mathcal T_h $}
  \right\}.
\]
For any $T\in\mathcal T_h $, let $ \lambda_1,\lambda_2,\ldots,\lambda_{d+1}$ be
the conventional barycentric coordinate functions defined on $T$. Then we denote
\begin{equation}
  \label{eq:def-Gamma-T}
  \Gamma(T):= S_1(T)+S_2(T)+\cdots+S_{d+1}(T),
\end{equation}
where
\[
  S_i(T) := \left( \prod_{j\neq i}\lambda_j \right) \text{span}
  \left\{
    \prod_j\lambda_j^{\alpha_j}:\sum_j\alpha_j=k,\alpha_i=0
  \right\},
  \quad i=1,2,\ldots,d+1.
\]

We define  the first  interpolation operator $ \Pi_h^0:
L^2(\cup_{F\in\mathcal F_h}F) \to V_h^{\text{c}} $  as follows: for any $ \mu \in
L^2(\cup_{F\in\mathcal F_h}F) $, $ \Pi_h^0\mu $ satisfies
\[
  \begin{cases}
    \Pi_h^0\mu(a) =
    \frac1{\#\omega_{a}} \sum_{T\in\omega_{a}}m_T(\mu) &
    \text{ if $a$ is an interior node of $ \mathcal T_h $, } \\
    \Pi_h^0\mu(a) = 0 &
    \text{ if $a$ is a boundary node of $ \mathcal T_h $.}
  \end{cases}
\]
Here $ \omega_a := \left\{T\in\mathcal T_h:a\text{ is a vertex of
$T$}\right\} $, $ \#\omega_{a}$ denotes the number of elements in $ \omega_{a}$, and
$m_T(\cdot):L^2(\partial T)\to\mathbb R$ is given by
\begin{equation}
  \label{eq:def-m_T}
  m_T(\mu) := \frac1{d+1} \sum_{F\in\mathcal F_T}
  \frac1{|F|} \int_F \mu \quad
  \text{for all $ \mu\in L^2(\partial T) $,}
\end{equation}
where $ \mathcal F_T := \left\{ F: F \text{ is a face of $T$} \right\} $, and $
\verti{F} $ denotes the $ (d-1) $-dimensional Hausdorff measure.

We proceed to define the second interpolation operator $ \Pi_h^1 $ on $
L^2(\Omega)\times L^2(\cup_{F\in\mathcal F_h}F) $. Let us start by defining a
local version $ \Pi_T^1 $ of $ \Pi_h^1 $ for all $ T\in\mathcal T_h $. For any $
(v,\mu)\in L^2(T)\times L^2(\partial T) $, by the
definition~\eqref{eq:def-Gamma-T} of $ \Gamma(T) $, it is easy to see that there
exists a unique $w_1\in\Gamma(T) $ such that
\[
  \int_F w_1 q = \int_F \mu q \quad
  \text{for all $ q\in P_k(F) $ and $ F\in\mathcal F_T $,}
\]
and that there exists a unique $ w_2 \in
(\prod_j\lambda_j)P_k(T) $ such that
\[
  \int_T w_2 q = \int_T(v-w_1)q \quad \text{for all $q\in P_k(T) $,}
\]
and then we define $ \Pi_T(v,\mu) := w_1 + w_2 $. Now we define $ \Pi_h^1 $ as
follows: for any $ (v,\mu) \in L^2(\Omega) \times L^2(\cup_{F\in\mathcal F_h}F)
$ and $ T\in\mathcal T_h $,
\[
  \Pi_h^1(v,\mu)|_T := \Pi_T^1(v,\mu).
\]

Finally, based on the operators $ \Pi_h^0$ and $ \Pi_h^1 $, we define the third
interpolation operator $ \Pi_h $ on $ L^2(\Omega)\times L^2(\cup_{F\in\mathcal
F_h}F) $ by
\begin{equation}
  \label{eq:def-Pi_h}
  \Pi_h(v,\mu) := \Pi_h^0\mu + \Pi_h^1(v-\Pi^0_h\mu,\mu-\Pi^0_h\mu)
  \quad \text{
    for all $ (v,\mu)\in L^2(\Omega)\times L^2(\cup_{F\in\mathcal F_h}F) $.
  }
\end{equation}

The following two lemmas show some properties of the above interpolation
operators.
\begin{lem}
  \label{lem:Pi_h1}
  For any $ (v,\mu) \in L^2(\Omega) \times L^2(\cup_{F \in \mathcal F_h}F) $, it
  holds
  \[
    \norm{\Pi_h^1(v,\mu)}_T \lesssim
    \norm{v}_T + h_T^{\frac 12}\norm{\mu}_{\partial T}
    \quad \text{for all $ T \in \mathcal T_h $.}
  \]
\end{lem}
\begin{lem}
  \label{lem:basic-Pi_h}
  For any $ (v,\mu) \in L^2(\Omega) \times L^2(\cup_{F \in \mathcal F_h}F) $
  and $ T \in \mathcal T_h $, it holds
  \begin{align}
    \big(\Pi_h(v,\mu), q\big)_T = (v, q)_T &
    \quad\text{for all $ q \in P_k(T) $,} \\
    \dual{\Pi_h(v,\mu),q}_{\partial T} = \dual{\mu,q}_{\partial T} &
    \quad\text{for all $ q \in P_k(F) $ and $ F \in \mathcal F_T$.}
  \end{align}
\end{lem}

 Lemma~\ref{lem:Pi_h1} follows from a standard scaling
argument, and Lemma~\ref{lem:basic-Pi_h} follows from   the definition of $ \Pi_h $.

\subsection{Error estimation for numerical flux}

In this subsection, we shall follow the basic idea in \cite{Gudi;2010} to give
an error estimate for the numerical flux. We stress that we only need to use the
following minimal regularity condition of the problem~\eqref{eq:model}:
\begin{equation}
  \label{eq:reg_cond}
  u\in H^1(\Omega) \text{ and } \bm\sigma\in H(\ddiv;\Omega).
\end{equation}

Define
\begin{equation}
  \label{eq:def-e_h}
  \bm e_h^{\bm\sigma} := \bm\sigma_h - \mathcal P_h^k \bm\sigma,
  \quad e_h^u: = u_h - \mathcal P_h^{k+1} u,
  \quad e_h^\lambda: = \lambda_h - \mathcal P_M u,
\end{equation}
where, $ \mathcal P_M: H^1(\Omega) \to M_h $ is given by
\[
  \sum_{T \in \mathcal T_h}
  \dual{\mathcal P_M v,\mu_h}_{\partial T} :=
  \sum_{T \in \mathcal T_h}
  \dual{v,\mu_h}_{\partial T} \quad
  \text{for all $ \mu_h \in M_h $,}
\]
namely,
\[
  (\mathcal P_M v)|_{\partial T} := \mathcal P_T^\partial (v|_T)
  \quad \text{for all $ T \in \mathcal T_h $.}
\]
In addition, by~\eqref{eq:model} and~\eqref{eq:hdg-1}, it is easy to verify that
\begin{equation}
  \label{eq:err-1}
  (\bm c\bm e_h^{\bm\sigma}, \bm\tau_h) + (e_h^u,\ddiv_h\bm\tau_h) -
  \sum_{T \in \mathcal T_h} \dual{e_h^\lambda, \bm\tau_h\cdot\bm n}_{\partial T}
  =
  \big(
    \bm c(I-\mathcal P_h^k)\bm\sigma, \bm\tau_h
  \big)
\end{equation}
for all $ \bm\tau_h \in \bm W_h $.

We introduce a semi-norm $ \vertiii{\cdot}_h: V_h\times M_h \times \bm W_h\to\mathbb R$ by
\begin{equation}\label{eq:def_vertiii_h}
  \vertiii{(v_h,\mu_h, \bm\tau_h)}_h: =
  \left(
    \norm{\bm\tau_h}_{\bm c}^2 + \sum_{T\in\mathcal T_h}
    \norm{\alpha_T^{\frac{1}{2}}(\mathcal P_T^\partial  v_h - \mu_h)}_{\partial T}^2
  \right)^{\frac{1}{2}}
\end{equation}
for all $(v_h,\mu_h,\bm\tau_h) \in V_h\times M_h \times \bm W_h $, where
\[
  \norm{\bm \tau}_{\bm c}: = (\bm c \bm \tau, \bm\tau)^{\frac{1}{2}}
  \quad\text{for all $ \bm \tau \in [L^2(\Omega)]^d$.}
\]
We also set
\[
  \verti{w}_{1,h} := \left(
    \sum_{T \in \mathcal T_h} \verti{w}_{1,T}^2
  \right)^\frac12
\]
for any
\(
  w \in L^2(\Omega)\) with $w|_T \in H^1(T) $ for all $ T \in \mathcal T_h $. 

Now we are ready to state 
the main result of this subsection.
\begin{thm}
  \label{thm:conv_sigma}
  It holds
  \begin{equation}\label{eq:conv_sigma-1}
    \vertiii{(e_h^u,e_h^{\lambda}, \bm e_h^{\bm\sigma})}_h
    \lesssim
    h\norm{(I-\mathcal P_h^k)f} +
    \norm{(I-\mathcal P_h^k)\bm\sigma} +
    \verti{(I-\mathcal P_h^{k+1})u}_{1,h},
  \end{equation}
  which implies
  \begin{equation}\label{eq:conv_sigma-2}
    \norm{\bm\sigma-\bm\sigma_h} \lesssim h
    \norm{(I-\mathcal P_h^k)f} +
    \norm{(I-\mathcal P_h^k)\bm\sigma} +
    \verti{(I-\mathcal P_h^{k+1})u}_{1,h}.
  \end{equation}
\end{thm}

By the estimate \eqref{eq:conv_sigma-2} and   standard approximation properties of the $ L^2
$-orthogonal projection, we readily obtain the corollary below.
\begin{coro}\label{coro1}
  Suppose that $ f \in H^s(\Omega) $, $ \bm\sigma \in H^{s+1}(\Omega) $, and $ u
  \in H^{s+2}(\Omega) $ for a nonnegative integer $ s $.  Then it holds
  \begin{equation}
    \norm{\bm\sigma - \bm\sigma_h} \lesssim h^{\min\{s+1,k+1\}}
    \left(
      \norm{f}_{s,\Omega} + \norm{\bm\sigma}_{s+1,\Omega} +
      \norm{u}_{s+2,\Omega}
    \right).
  \end{equation}
\end{coro}

To prove Theorem \ref{thm:conv_sigma}, we need some lemmas below.

\begin{lem}
  \label{lem:key}
  It holds
  \begin{equation}
    \label{auxiliary_ineq}
    \verti{e_h^u}_{1,h} \lesssim
    \vertiii{ (e_h^u, e_h^\lambda, \bm e_h^{\bm\sigma}) }_h +
    \norm{ (I-\mathcal P_h^k)\bm\sigma }.
  \end{equation}
\end{lem}
\begin{proof}
  By~\eqref{eq:err-1} and integration by parts, we get that
  \[
    \sum_{ T \in \mathcal T_h } (\bm\nabla e_h^u, \bm\tau_h)_T =
    ( \bm c\bm e_h^{\bm\sigma},\bm\tau_h ) +
    \sum_{ T \in \mathcal T_h } \dual{
      \mathcal P_T^\partial e_h^u - e_h^\lambda,
      \bm\tau_h\cdot\bm n
    }_{\partial T} -
    \big(
      \bm c(I-\mathcal P_h^k)\bm\sigma,
      \bm\tau_h 
    \big)
  \]
  for all $ \bm\tau_h \in \bm W_h $. Taking $ \bm\tau_h := \bm\nabla_h e_h^u $
  (i.e., $ \bm\tau_h|_T := \bm\nabla (e_h^u|_T) $ for all $ T \in \mathcal T_h $)
  in the above equation, and using the Cauchy-Schwarz inequality, we obtain
  \begin{align*}
    \verti{e_h^u}_{1,h}^2
    ={} &
    (\bm c \bm e_h^{\bm\sigma}, \bm\nabla_h e_h^u) +
    \sum_{T \in \mathcal T_h}
    \dual{
      \mathcal P_T^\partial e_h^u - e_h^\lambda,
      \bm\nabla e_h^u\cdot\bm n
    }_{\partial T} - 
    \big(
      \bm c(I-\mathcal P_h^k)\bm\sigma,
      \bm\nabla_h e_h^u
    \big) \\
    \lesssim{} &
    \norm{\bm e_h^{\bm\sigma}} \verti{e_h^u}_{1,h} +
    \sum_{T \in \mathcal T_h}
    \norm{\mathcal P_T^\partial e_h^u - e_h^\lambda}_{\partial T}
    \norm{\bm\nabla e_h^u}_{\partial T} +
    \norm{(I-\mathcal P_h^k)\bm\sigma} \verti{e_h^u}_{1,h} \\
    \lesssim{} &
    \left(
      \norm{\bm e_h^{\bm\sigma}} +
      \norm{(I-\mathcal P_h^k)\bm\sigma}
    \right)
    \verti{e_h^u}_{1,h} +
    \left(
      \sum_{T \in \mathcal T_h}
      h_T^{-1}\norm{\mathcal P_T^\partial e_h^u - e_h^\lambda}_{\partial T}^2
    \right)^{\frac 12}
    \left(
      \sum_{T \in \mathcal T_h}
      h_T\norm{\bm\nabla e_h^u}_{\partial T}^2
    \right)^{\frac 12}.
  \end{align*}
  Noting that a standard scaling argument gives
  \[
    h_T\norm{\bm\nabla e_h^u}_{\partial T}^2 \lesssim
    \verti{e_h^u}_{1,T}^2,
  \]
  it follows that
  \[
    \verti{e_h^u}_{1,h}
    \lesssim
    \norm{ \bm e_h^{\bm\sigma} } + \norm{ (I-\mathcal P_h^k) \bm\sigma } +
    \left(
      \sum_{T \in \mathcal T_h}
      h_T^{-1} \norm{
        \mathcal P_T^\partial e_h^u - e_h^\lambda
      }_{\partial T}^2
    \right)^\frac12.
  \]
  By the definition~\eqref{eq:def_vertiii_h} of $ \vertiii{\cdot}_h $, we readily
  obtain \eqref{auxiliary_ineq}, and thus complete the proof.
\end{proof}

\begin{lem}\label{lem:basic}
  It holds 
  \begin{equation}
    \left(\sum_{T\in\mathcal T_h}
    h_T^{-2}\norm{e_h^u-m_T(e_h^{\lambda})}_T^2
    +h_T^{-1}\norm{e_h^{\lambda}-m_T(e_h^{\lambda})}_{\partial T}^2
  \right)^\frac12
  \lesssim
  \vertiii{(e_h^u,e_h^{\lambda},\bm e_h^{\bm\sigma})}_h +
  \norm{(I-\mathcal P_h^k)\bm\sigma}.
\end{equation}
\end{lem}
\begin{proof}
  By the definition~\eqref{eq:def-m_T} of $m_T(\cdot) $, a standard scaling
  argument yields
  \[
    \norm{e_h^u-m_T(e_h^u)}_T \lesssim h_T\verti{e_h^u}_{1,T},
  \]
  and a straightforward computation gives
  \begin{align*}
    \norm{m_T(e_h^u)-m_T(e_h^\lambda)}_T
    &= \frac1{d+1} \norm{
      \sum_{F\in\mathcal F_T}
      \frac1{|F|} \int_F(\mathcal P_T^\partial e_h^u-e_h^\lambda)
    }_T \\
    &\leqslant \frac1{d+1} \sum_{F \in \mathcal F_T}
    \norm{
      \frac1{|F|} \int_F(\mathcal P_T^\partial e_h^u-e_h^\lambda)
    }_T \\
    &\lesssim h_T^\frac12 \sum_{F \in \mathcal F_T}
    \norm{\mathcal P_T^\partial e_h^u - e_h^\lambda}_F \\
    &\lesssim h_T^\frac12 \norm{\mathcal P_T^\partial e_h^u-e_h^\lambda}_{\partial T}.
  \end{align*}
  Consequently, by the definition \eqref{eq:def_vertiii_h} of
  $ \vertiii{\cdot}_h $, we obtain
  \begin{align*}
    \sum_{T\in\mathcal T_h} h_T^{-2} \norm{e_h^u-m_T(e_h^{\lambda})}_T^2
    &\lesssim\sum_{T\in\mathcal T_h}h_T^{-2}\left(\norm{e_h^u-m_T(e_h^u)}^2_T +
    \norm{m_T(e_h^u)-m_T(e_h^\lambda)}^2_T\right) \\
    &\lesssim\sum_{T\in\mathcal T_h}\left(
      \verti{e_h^u}_{1,T}^2 +
      h_T^{-1}\norm{\mathcal P_T^\partial e_h^u-e_h^{\lambda}}^2_{\partial T}
    \right) \\
    &\lesssim\vertiii{(e_h^u,e_h^{\lambda},\bm e_h^{\bm\sigma})}^2_h +
    \sum_{T \in \mathcal T_h} \verti{e_h^u}_{1,T}^2,
  \end{align*}
  which, together with Lemma~\ref{lem:key}, indicates
    \[
    \left(
      \sum_{T\in\mathcal T_h} h_T^{-2}\norm{e_h^u-m_T(e_h^\lambda)}^2_T
    \right)^\frac12
    \lesssim \vertiii{(e_h^u,e_h^\lambda,\bm e_h^{\bm\sigma})}_h +
    \norm{(I-\mathcal P_h^k)\bm\sigma}.
  \]

  To complete the proof, the thing left is to show, for any $ T \in
  \mathcal T_h $,
  \begin{equation}\label{eq:999}
    h_T^{-1}\norm{e_h^{\lambda}-m_T(e_h^{\lambda})}_{\partial T}^2\lesssim
    h_T^{-2}\norm{e_h^u-m_T(e_h^\lambda)}_T^2+\norm{\bm e_h^{\bm\sigma}}^2_T
    +\norm{(I-\mathcal P_T^k)\bm\sigma}_T^2.
  \end{equation}
 In fact, by~\eqref{eq:err-1} and integration by
  parts, we get, for any $T\in\mathcal T_h $ and $ \bm\tau\in\bm W(T) $,
  \begin{equation}\label{eq:777}
    \bint{e_h^{\lambda}-m_T(e_h^{\lambda}),\bm\tau\cdot\bm n}{\partial T} =
    (\bm c\bm e_h^{\bm\sigma},\bm\tau)_T+(e_h^u-m_T(e_h^{\lambda}),\ddiv\bm\tau)_T
    -\big(\bm c (I-\mathcal P_T^k)\bm\sigma,\bm\tau\big)_T.
  \end{equation}
  Let us first show that \eqref{eq:999} holds in the case of $k=0$. Evidently,
  there exists a unique $v\in P_1(T) $ such that
  \begin{equation}
    \label{eq:v}
    \int_Fv = \int_F\big(e_h^\lambda-m_T(e_h^\lambda)\big)
    \quad\text{for all $ F \in \mathcal F_T $.}
  \end{equation}
  Taking $ \bm\tau:=\bm\nabla v$ in \eqref{eq:777}, and noting the fact that
  $ \ddiv\bm\tau = 0$, we easily get 
  \[
    \bint{e_h^\lambda-m_T(e_h^\lambda),\bm\nabla v\cdot\bm n}{\partial T}
    \lesssim \norm{\bm\nabla v}_T
    \left(
      \norm{\bm e_h^{\bm\sigma}}_T
      +\norm{(I-\mathcal P_T^k)\bm\sigma}_T
    \right).
  \]
  From~\eqref{eq:v} and integration by parts it follows
  \[
    \dual{
      e_h^\lambda-m_T(e_h^\lambda),\bm\nabla v\cdot\bm n
    }_{\partial T} =
    \dual{
      v, \bm\nabla v \cdot \bm n
    }_{\partial T} =
    \norm{\bm\nabla v}^2_T.
  \]
  The above two estimates imply
  \[
    \norm{\bm\nabla v}_T \lesssim
    \norm{\bm e_h^{\bm\sigma}}_T +
    \norm{(I-\mathcal P_T^k)\bm\sigma}_T.
  \]
  Since~\eqref{eq:v} yields $ m_T(v) = 0 $,  a standard
  scaling argument gives
  \[
    h_T^{-1}\norm{v}_{\partial T}^2 =
    h_T^{-1}\norm{v-m_T(v)}_{\partial T}^2
    \lesssim \norm{\bm\nabla v}_T^2.
  \]
  We note that~\eqref{eq:v} also leads to
  \[
    h_T^{-1} \norm{e_h^\lambda - m_T(e_h^\lambda)}_{\partial T}^2
    \lesssim h_T^{-1} \norm{v}_{\partial T}^2.
  \]
 As a result,~\eqref{eq:999}  with $k=0$ follows from the above three estimates.

  Next we consider the case of $k\geqslant 1 $. By
  the well-known properties of the BDM elements~\cite{BDM;1985}, there exists a
  $ \bm\tau\in\bm W(T) $ such that
  \[
    \dual{
      e_h^\lambda-m_T(e_h^\lambda),
      \bm\tau\cdot\bm n
    }_{\partial T} =
    \norm{e_h^\lambda-m_T(e_h^\lambda)}^2_{\partial T} \text{ and }
    \norm{\bm\tau}_T \lesssim h_T^\frac12
    \norm{e_h^\lambda-m_T(e_h^\lambda)}_{\partial T}.
  \]
  Taking the above $ \bm\tau$ in~\eqref{eq:777}, and using   standard inverse
  estimates, we obtain
  \[
    \norm{e_h^\lambda-m_T(e_h^\lambda)}_{\partial T}\lesssim
    h_T^\frac12\norm{\bm e_h^{\bm\sigma}}_T +
    h_T^{-\frac12}\norm{e_h^u-m_T(e_h^\lambda)}_T +
    h_T^\frac12\norm{(I-\mathcal P_T^k)\bm\sigma}_T.
  \]
  This implies~\eqref{eq:999}, and thus completes the proof.
\end{proof}

Define $ \bm\eta_h\in\bm W_h $ by
\begin{equation}
  \label{eq:def_eta_h}
  (\bm\eta_h,\bm\tau)_T := -(e_h^u,\ddiv\bm\tau)_T +
  \dual{e_h^\lambda,\bm\tau\cdot\bm n}_{\partial T}\quad \text{for all $ \bm\tau\in\bm W(T) $ and $T\in\mathcal T_h $.}
\end{equation}

\begin{lem}
  \label{lem:lbjlxy-1}
  It holds
  \begin{equation}
    \label{eq:lsj-1}
    \big(\bm\nabla \Pi_h(e_h^u,e_h^\lambda) - \bm\eta_h, \bm q)_T = 0
  \end{equation}
  for all $ \bm q \in [P_k(T)]^d$ and $ T \in \mathcal T_h $. Moreover,
  \begin{equation}
    \norm{\bm\eta_h} \lesssim
    \vertiii{(e_h^u,e_h^\lambda,\bm e_h^{\bm\sigma})}_h +
    \norm{(I-\mathcal P_h^k)\bm\sigma}. \label{eq:lxy22} \end{equation}
\end{lem}
\begin{proof}
   The relation~\eqref{eq:lsj-1} follows from \eqref{eq:def_eta_h}. In what follows we show~\eqref{eq:lxy22}. Taking $
  \bm\tau := \bm\eta_h|_T$ in~\eqref{eq:def_eta_h}, and using integration by
  parts and   inverse estimates, we obtain
  \begin{align*}
    (\bm\eta_h,\bm\eta_h)_T
    &= -(e_h^u-m_T(e_h^\lambda),\ddiv\bm\eta_h)_T +
    \bint{e_h^\lambda-m_T(e_h^\lambda),\bm\eta_h\cdot\bm n}{\partial T}\\
    &\lesssim
    h_T^{-1}\norm{e_h^u-m_T(e_h^\lambda)}_T \norm{\bm\eta_h}_T +
    \norm{e_h^\lambda-m_T(e_h^\lambda)}_{\partial T}
    \norm{\bm\eta_h}_{\partial T},
  \end{align*}
 which, together with the fact that $ \norm{\bm\eta_h}_{\partial T}
  \lesssim h_T^{-\frac 12} \norm{\bm\eta_h}_T $, implies
  \begin{align*}
    (\bm\eta_h,\bm\eta_h)_T
    &\lesssim\left(h_T^{-1}\norm{e_h^u-m_T(e_h^\lambda)}_T +
    h_T^{-\frac12}\norm{e_h^\lambda-m_T(e_h^\lambda)}_{\partial
    T}\right)\norm{\bm\eta_h}_T.
  \end{align*}
 Then it follows 
  \[
    \norm{\bm\eta_h}_T\lesssim h_T^{-1}\norm{e_h^u-m_T(e_h^\lambda)}_T +
    h_T^{-\frac12}\norm{e_h^\lambda-m_T(e_h^\lambda)}_{\partial T},
  \]
which, together with Lemma~\ref{lem:basic}, yields the estimate~\eqref{eq:lxy22}. This completes the proof.
\end{proof}

\begin{lem}
  \label{lem:lbjlxy}
  It holds
  \begin{align}
    \verti{\Pi_h(e_h^u,e_h^{\lambda})}_{1,\Omega} &
    \lesssim \vertiii{
      (e_h^u,e_h^{\lambda}, \bm e_h^{\bm\sigma})
    }_h +
    \norm{(I-\mathcal P_h^k)\bm\sigma}, \label{eq:lxy3} \\
    \left(
      \sum_{T\in\mathcal T_h} h_T^{-2}
      \norm{e_h^u-\Pi_h(e_h^u,e_h^{\lambda})}^2_T
    \right)^\frac12 &
    \lesssim \vertiii{
      (e_h^u,e_h^{\lambda}, \bm e_h^{\bm\sigma})
    }_h +
    \norm{(I-\mathcal P_h^k)\bm\sigma}. \label{eq:lxy4}
  \end{align}
\end{lem}
\begin{proof}
  Let us first show~\eqref{eq:lxy3}. We have, for any $ T \in \mathcal T_h $,
  \begin{align*}
    {} &
    \verti{\Pi_h(e_h^u,e_h^\lambda)}_{1,T} \\
    ={} &
    \verti{
      \Pi_h^0e_h^\lambda+\Pi_h^1(e_h^u -
      \Pi_h^0e_h^\lambda,e_h^\lambda-\Pi_h^0e_h^\lambda)
    }_{1,T}
    \quad \text{(by~\eqref{eq:def-Pi_h})} \\
    \leqslant{} &
    \verti{\Pi_h^0e_h^\lambda}_{1,T} +
    \verti{
      \Pi_h^1(e_h^u-\Pi_h^0e_h^\lambda,
      e_h^\lambda-\Pi_h^0e_h^\lambda)
    }_{1,T} \\
    \leqslant{} &
    \verti{\Pi_h^0e_h^\lambda-m_T(e_h^\lambda)}_{1,T} +
    \verti{
      \Pi_h^1(e_h^u-\Pi_h^0e_h^\lambda,
      e_h^\lambda-\Pi_h^0e_h^\lambda)
    }_{1,T} \\
    \lesssim{} &
    h_T^{-1}\norm{\Pi_h^0e_h^\lambda-m_T(e_h^\lambda)}_T +
    h_T^{-1}\norm{\Pi_h^1(e_h^u-\Pi_h^0e_h^\lambda,e_h^\lambda-\Pi_h^0e_h^\lambda)}_T
    \quad \text{(by inverse estimates)} \\
    \lesssim{} &
    h_T^{-1}\norm{\Pi_h^0e_h^\lambda-m_T(e_h^\lambda)}_T +
    h_T^{-1}\norm{e_h^u-\Pi_h^0e_h^\lambda}_T +
    h_T^{-\frac12}\norm{e_h^\lambda-\Pi_h^0e_h^\lambda}_{\partial T}
    \quad \text{(by Lemma~\ref{lem:Pi_h1})} \\
    \lesssim{} &
    h_T^{-1}\norm{\Pi_h^0e_h^\lambda-m_T(e_h^\lambda)}_T +
    h_T^{-1}\norm{e_h^u-m_T(e_h^\lambda)}_T +
    h_T^{-\frac12}\norm{e_h^\lambda-\Pi_h^0e_h^\lambda}_{\partial T}.
  \end{align*}
  From the estimate
  \[
    \norm{\Pi_h^0 e_h^\lambda - m_T(e_h^\lambda)}_T \lesssim
    h_T^\frac12 \norm{\Pi_h^0 e_h^\lambda - m_T(e_h^\lambda)}_{\partial T},
  \]
  it follows 
  \begin{align*}
    {} &
    \verti{\Pi_h(e_h^u,e_h^\lambda)}_{1,T} \\
    \lesssim{} &
    h_T^{-\frac12}\norm{\Pi_h^0e_h^\lambda-m_T(e_h^\lambda)}_{\partial T} +
    h_T^{-1}\norm{e_h^u-m_T(e_h^\lambda)}_T +
    h_T^{-\frac12}\norm{e_h^\lambda-\Pi_h^0e_h^\lambda}_{\partial T} \\
    \lesssim{} &
    h_T^{-1}\norm{e_h^u-m_T(e_h^\lambda)}_T +
    h_T^{-\frac12}\norm{e_h^\lambda-m_T(e_h^\lambda)}_{\partial T} +
    h_T^{-\frac12}\norm{\Pi_h^0e_h^\lambda-m_T(e_h^\lambda)}_{\partial T},
  \end{align*}
  and then, by Lemma~\ref{lem:basic}, it holds
  \begin{align}
    {} &
    \verti{\Pi_h(e_h^u,e_h^\lambda)}_{1,\Omega} =
    \left(
      \sum_{T\in\mathcal T_h}\verti{\Pi_h(e_h^u,e_h^\lambda)}^2_{1,T}
    \right)^\frac12 \nonumber \\
    \lesssim{} &
    \left(
      \sum_{T\in\mathcal T_h} \left(
        h_T^{-2}\norm{e_h^u-m_T(e_h^\lambda)}^2_T +
        h_T^{-1}\norm{e_h^\lambda-m_T(e_h^\lambda)}^2_{\partial T} +
        h_T^{-1}\norm{\Pi_h^0e_h^\lambda-m_T(e_h^\lambda)}^2_{\partial T}
      \right)
  \right)^\frac12\nonumber \\
  \lesssim{} &
  \vertiii{(e_h^u,e_h^\lambda,\bm e_h^{\bm\sigma})}_h +
  \norm{(I-\mathcal P_h^k)\bm\sigma} +
  \left(
    \sum_{T\in\mathcal T_h} h_T^{-1} \norm{\Pi_h^0e_h^\lambda-m_T(e_h^\lambda)}^2_{\partial T}
  \right)^\frac12.
  \label{eq:3956}
\end{align}
To obtain~\eqref{eq:lxy3}, it remains to show
\[
  \left(
    \sum_{T\in\mathcal T_h}
    h_T^{-1} \norm{\Pi_h^0e_h^\lambda-m_T(e_h^\lambda)}^2_{\partial T}
  \right)^\frac12 \lesssim
  \vertiii{(e_h^u,e_h^\lambda,\bm e_h^{\bm\sigma})}_h +
  \norm{(I-\mathcal P_h^k)\bm\sigma}.
\]
By the definition of $ \Pi_h^0 $,  we have, for any $ T \in \mathcal T_h $ such
that $ \overline T \subset \Omega $,
\begin{align*}
  \norm{\Pi_h^0e_h^\lambda-m_T(e_h^\lambda)}_{\partial T}
  &\lesssim h_T^{\frac{d-1}{2}} \sum_{a\in\mathcal N(T)}
  |\Pi_h^0e_h^{\lambda}(a)-m_T(e_h^{\lambda})| \\
  &\lesssim h_T^{\frac{d-1}{2}} \sum_{a\in\mathcal N(T)}
  \sum_{
    \substack{
      T_1,T_2\in\omega_a \\
      |\partial T_1\cap\partial T_2|\neq 0
    }
  } |m_{T_1}(e_h^{\lambda})-m_{T_2}(e_h^{\lambda})| \\
  &\lesssim h_T^{\frac{d-1}{2}}\sum_{a\in\mathcal N(T)}
  \sum_{
    \substack{
      T_1,T_2\in\omega_a \\
      |\partial T_1\cap\partial T_2|\neq 0
    }
  }
  h_T^{-\frac{d-1}2} \norm{
    m_{T_1}(e_h^{\lambda})-m_{T_2}(e_h^{\lambda})
  }_{\partial T_1 \cap \partial T_2} \\
  &\lesssim \sum_{a\in\mathcal N(T)} \sum_{
    \substack{
      T_1,T_2\in\omega_a \\
      |\partial T_1\cap\partial T_2|\neq 0
    }
  }
  \left(
    \norm{e_h^\lambda-m_{T_1}(e_h^\lambda)}_{\partial T_1\cap\partial T_2} +
    \norm{e_h^\lambda-m_{T_2}(e_h^\lambda)}_{\partial T_1\cap\partial T_2}
  \right) \\
  &\lesssim \sum_{a\in\mathcal N(T)} \sum_{T'\in\omega_a}
  \norm{e_h^{\lambda}-m_{T'}(e_h^{\lambda})}_{\partial T'},
\end{align*}
where $\mathcal N(T)$ denotes the set of vertexes of $T$, and we recall the
definition of $\omega_a$ by
\[
  \omega_a :=
  \left\{
    T \in \mathcal T_h:
    \text{ $a$ is a vertex of $T$}
  \right\}.
\]
Similarly, for any $ T \in \mathcal T_h $ such that $ \partial T \cap
\partial\Omega \neq \emptyset $, we also have
\begin{align*}
  \norm{\Pi_h^0e_h^\lambda-m_T(e_h^\lambda)}_{\partial T}
  \lesssim \sum_{a\in\mathcal N(T)} \sum_{T'\in\omega_a}
  \norm{e_h^{\lambda}-m_{T'}(e_h^{\lambda})}_{\partial T'}.
\end{align*}
As a consequence, from Lemma~\ref{lem:basic} it follows
\begin{align*}
  {} &
  \left(
    \sum_{T\in\mathcal T_h}
    h_T^{-1} \norm{\Pi_h^0e_h^\lambda-m_T(e_h^\lambda)}_{\partial T}^2
  \right)^{\frac 12} \\
  \lesssim{} &
  \left(
    \sum_{T \in \mathcal T_h}
    \sum_{a\in\mathcal N(T)} \sum_{T'\in\omega_a} h_{T'}^{-1}
    \norm{e_h^{\lambda}-m_{T'}(e_h^{\lambda})}_{\partial T'}^2
  \right)^\frac12 \\
  \lesssim{} &
  \left(
    \sum_{T \in \mathcal T_h} h_T^{-1}
    \norm{e_h^{\lambda}-m_T(e_h^{\lambda})}_{\partial T}^2
  \right)^\frac12 \\
  \lesssim{} &
  \vertiii{(e_h^u,e_h^\lambda,\bm
  e_h^{\bm\sigma})}_h + \norm{(I-\mathcal P_h^k)\bm\sigma}.
\end{align*}
This completes the proof of~\eqref{eq:lxy3}.

Next, let us show~\eqref{eq:lxy4}. By Lemma~\ref{lem:basic-Pi_h} we have
\[
  \mathcal P_T^0 e_h^u = \mathcal P_T^0 \Pi_h(e_h^u,e_h^\lambda)
  \quad\text{for all $T\in\mathcal T_h $.}
\]
Using standard approximation properties of the $L^2$-orthogonal
projection, we obtain 
\begin{align*}
  \norm{e_h^u-\Pi_h(e_h^u,e_h^\lambda)}_T
  &= \norm{(I-\mathcal P_T^0)e_h^u - (I-\mathcal P_T^0)\Pi_h(e_h^u,e_h^\lambda)}_T \\
  &\leqslant \norm{(I-\mathcal P_T^0)e_h^u}_T + \norm{(I-\mathcal P_T^0)\Pi_h(e_h^u,e_h^\lambda)}_T \\
  &\lesssim h_T\verti{e_h^u}_{1,T}+h_T\verti{\Pi_h(e_h^u,e_h^\lambda)}_{1,T}
\end{align*}
for all $T\in\mathcal T_h $. Then, from Lemma~\ref{lem:key} and
\eqref{eq:lxy3} it follows
\begin{align*}
  \left(
    \sum_{T\in\mathcal T_h} h_T^{-2} \norm{e_h^u-\Pi_h(e_h^u,e_h^\lambda)}_T^2
  \right)^\frac12
  &\lesssim
  \left(
    \sum_{T\in\mathcal T_h} \left(
      \verti{e_h^u}_{1,T}^2 +
      \verti{\Pi_h(e_h^u,e_h^\lambda)}^2_{1,T}
    \right)
  \right)^\frac12 \\
  &\lesssim \vertiii{
    (e_h^u,e_h^\lambda,\bm e_h^{\bm\sigma})
  }_h + \norm{(I-\mathcal P_h^k)\bm\sigma}.
\end{align*}
This completes the proof.
\end{proof}

Finally, we are in a position to prove
Theorem~\ref{thm:conv_sigma}. \\
{\bf Proof of Theorem~\ref{thm:conv_sigma}.}
  By~\eqref{eq:err-1},~\eqref{eq:hdg-2} and~\eqref{eq:hdg-3},   straightforward
  algebraic calculations show
  \begin{align*}
    (\bm c\bm e_h^{\bm\sigma}, \bm e_h^{\bm\sigma}) +
    (e_h^u,\ddiv_h\bm e_h^{\bm\sigma}) -
    \bndint{e_h^\lambda, \bm e_h^{\bm\sigma}\cdot\bm n} &=
    \big(\bm c (I-\mathcal P_h^k)\bm\sigma, \bm e_h^{\bm\sigma}\big), \\
    -(e_h^u,\ddiv_h\bm e_h^{\bm\sigma}) +
    \sum_{T\in\mathcal T_h}
    \dual{\alpha_T(\mathcal P_T^\partial e_h^u-e_h^{\lambda}),e_h^u}_{\partial T}
    &=
    (f,e_h^u) + (e_h^u,\ddiv_h\mathcal P_h^k\bm\sigma) -
    \sum_{T\in\mathcal T_h}
    \dual{\alpha_T\mathcal P_T^\partial(\mathcal P_T^{k+1}u-u),e_h^u}_{\partial T}, \\
    \sum_{T\in\mathcal T_h}
    \dual{
      \bm e_h^{\bm\sigma}\cdot\bm n -
      \alpha_T(\mathcal P_T^\partial e_h^u-e_h^{\lambda}),
      e_h^{\lambda}
    }_{\partial T}
    &=
    -\sum_{T\in\mathcal T_h}
    \dual{
      \mathcal P_T^k\bm\sigma\cdot\bm n -
      \alpha_T\mathcal P_T^\partial(\mathcal P_T^{k+1}u-u),
      e_h^{\lambda}
    }_{\partial T}.
  \end{align*}
  Adding the above three equations, we easily get
  \begin{equation}
    \label{eq:7321}
    \vertiii{(e_h^u,e_h^{\lambda},\bm e_h^{\bm\sigma})}_h^2 =
    \mathbb I_1 + \mathbb I_2 + \mathbb I_3,
  \end{equation}
  where
  \begin{align*}
    \mathbb I_1 &:= \big(
      \bm c (I-\mathcal P_h^k)\bm\sigma, \bm e_h^{\bm\sigma}
    \big),\\
    \mathbb I_2 &:= -\sum_{T\in\mathcal T_h}
    \dual{\alpha_T(\mathcal P_T^{k+1}u-u),\mathcal P_T^\partial e_h^u-e_h^{\lambda}}_{\partial T},\\
    \mathbb I_3 &:= (f,e_h^u) + (e_h^u,\ddiv_h\mathcal P_h^k\bm\sigma) -
    \sum_{T\in\mathcal T_h}
    \dual{\mathcal P_T^k\bm\sigma\cdot\bm n,e_h^{\lambda}}_{\partial T}.
  \end{align*}

  In light of the definition~\eqref{eq:def_vertiii_h} of $ \vertiii{\cdot}_h $ and the
  fact that $ \alpha_T = h_T^{-1}$, we have
  \begin{equation}\label{eq:2000}
    \mathbb I_1+\mathbb I_2 \lesssim
    \left(
      \norm{(I-\mathcal P_h^k)\bm\sigma} +
      \left(
        \sum_{T\in\mathcal T_h}
        h_T^{-1} \norm{(I-\mathcal P_T^{k+1})u}^2_{\partial T}
      \right)^\frac12
    \right)
    \vertiii{(e_h^u,e_h^\lambda,\bm e_h^{\bm\sigma})}_h.
  \end{equation}
  By the definition \eqref{eq:def_eta_h} of $ \bm\eta_h $, we obtain
  \[
    \mathbb I_3 = (f, e_h^u) - (\bm\eta_h,\bm\sigma)
    = \big(f, e_h^u-\Pi_h(e_h^u,e_h^\lambda)\big) +
    \big(f,\Pi_h(e_h^u,e_h^\lambda)\big) -
    (\bm\eta_h,\bm\sigma).
  \]
  Since $-\ddiv\bm\sigma=f\in L^2(\Omega) $ and $ \Pi_h(e_h^u,e_h^\lambda)\in
  H_0^1(\Omega) $, using integration by parts, we get
  \[
    \big(f,\Pi_h(e_h^u,e_h^\lambda)\big) =
    \big(-\ddiv\bm\sigma,\Pi_h(e_h^u,e_h^\lambda)\big) =
    \big(\bm\sigma,\bm\nabla\Pi_h(e_h^u,e_h^\lambda)\big).
  \]
  The above two equations indicate
  \begin{align*}
    \mathbb I_3
    &=
    \big(f,e_h^u-\Pi_h(e_h^u,e_h^\lambda)\big) +
    \big(\bm\sigma,\bm\nabla\Pi_h(e_h^u,e_h^\lambda)-\bm\eta_h\big) \\
    &=
    \big((I-\mathcal P_h^k)f,e_h^u-\Pi_h(e_h^u,e_h^\lambda)\big) +
    \big((I-\mathcal P_h^k)\bm\sigma,\bm\nabla\Pi_h(e_h^u,e_h^\lambda)-\bm\eta_h\big)
    \quad \text{(by Lemmas~\ref{lem:basic-Pi_h} and \ref{lem:lbjlxy-1})} \\
    &\leqslant
    \norm{(I-\mathcal P_h^k) f} \norm{e_h^u-\Pi_h(e_h^u,e_h^\lambda)} +
    \norm{(I-\mathcal P_h^k)\bm\sigma}
    \norm{\bm\nabla\Pi_h(e_h^u,e_h^\lambda)-\bm\eta_h} \\
    &\leqslant
    \norm{(I-\mathcal P_h^k) f} \norm{e_h^u-\Pi_h(e_h^u,e_h^\lambda)} +
    \norm{(I-\mathcal P_h^k)\bm\sigma} \big(
      \verti{\Pi_h(e_h^u,e_h^\lambda)}_{1,\Omega} + \norm{\bm\eta_h}
    \big).
  \end{align*}
  This, together with Lemmas~\ref{lem:lbjlxy-1} and~\ref{lem:lbjlxy},  implies
  \begin{equation}
    \mathbb I_3 \lesssim
    \left(
      h\norm{(I-\mathcal P_h^k)f} +
      \norm{(I-\mathcal P_h^k)\bm\sigma}
    \right)
    \left(
      \vertiii{(e_h^u,e_h^\lambda,\bm e_h^{\bm\sigma})}_h +
      \norm{(I-\mathcal P_h^k)\bm\sigma}
    \right).\label{eq:2001}
  \end{equation}
  Finally, using \eqref{eq:7321}-\eqref{eq:2001} and Young's inequality
  \[
    ab \leqslant \epsilon a^2 + \frac{1}{4\epsilon}b^2
    \quad\text{for all $ \epsilon>0$,}
  \]
  we easily obtain
  \[
    \vertiii{(e_h^u,e_h^\lambda,\bm e_h^{\bm\sigma})}_h^2
    \lesssim h^2 \norm{(I-\mathcal P_h^k)f}^2 +
    \norm{(I-\mathcal P_h^k)\bm\sigma}^2 +
    \sum_{T \in \mathcal T_h} h_T^{-1}
    \norm{(I-\mathcal P_T^{k+1})u}_{\partial T}^2.
  \]
  Consequently,~\eqref{eq:conv_sigma-1} follows directly from the following
  standard estimate:
  \[
    h_T^{-1} \norm{(I-\mathcal P_T^{k+1})u}_{\partial T}^2
    \lesssim \verti{(I-\mathcal P_T^{k+1})u}_{1,T}^2 \quad
    \text{for all $ T \in \mathcal T_h $.}
  \]
  Since \eqref{eq:conv_sigma-2} is a direct consequence of
  \eqref{eq:conv_sigma-1}, the proof of Theorem~\ref{thm:conv_sigma} is
  finished. \hfill\ensuremath{\blacksquare}

\subsection{Error estimation for numerical potential}

Similarly to \cite{projection_based_hdg}, we shall use Aubin-Nitsche's technique
of duality argument to derive the error estimation for the numerical potential
$u_h $. Let us introduce the following dual problem:
\begin{equation}
  \label{eq:dual}
  \left\{
  \begin{array}{rll}
    \bm c\bm\Phi - \bm\nabla\phi \! & \!\! = 0 & \text{in $ \Omega$,} \\
    \ddiv \bm\Phi \! & \!\! = -e_h^u & \text{in $ \Omega$,} \\
    \phi   \! & \!\! = 0 & \text{on $ \partial\Omega$},
  \end{array}
  \right.
\end{equation}
where, as defined in \eqref{eq:def-e_h}, $ e_h^u := u_h - \mathcal P_h^{k+1} u
$. We stress that, in the following analysis, we only use the following minimal
regularity condition of the dual problem~\eqref{eq:dual}:
\begin{equation}
  \label{eq:reg-cond-dual}
  \phi \in H_0^1(\Omega) \quad \text{and} \quad \bm\Phi \in H(\ddiv;\Omega).
\end{equation}

 The main result of this subsection is the following theorem:
\begin{thm}
  \label{thm:conv_u}
  It holds
  \begin{equation}
    \label{eq:conv_u}
    \begin{split}
      \norm{e_h^u} & \lesssim
      h\vertiii{(e_h^u,e_h^\lambda,\bm e_h^{\bm\sigma})}_h +
      h\norm{(I-\mathcal P_h^k)\bm\sigma} +
      \norm{(I-\mathcal P_h^0)\bm\Phi}^\frac12
      \Big(
        \vertiii{(e_h^u,e_h^\lambda,\bm e_h^{\bm\sigma})}_h +
        \norm{(I-\mathcal P_h^k)\bm\sigma}
      \Big)^\frac12 \\
      & \quad {} +
      \norm{(I-\mathcal P_h^{k+1})f}^\frac12 \norm{(I-\mathcal P_h^{k+1})\phi}^\frac12 +
      \left(
        \vertiii{(e_h^u,e_h^\lambda,\bm e_h^{\bm\sigma})}_h +
        \verti{(I-\mathcal P_h^{k+1})u}_{1,h}
      \right)^\frac12
      \verti{(I-\mathcal P_h^{k+1})\phi}_{1,h}^\frac12.
    \end{split}
  \end{equation}
\end{thm}
\begin{proof}
  Since $ \bm\Phi \in H(\ddiv;\Omega) $ and $ \Pi_h(e_h^u,e_h^\lambda) \in
  H_0^1(\Omega) $, using integration by parts gives
  \[
    -\big(
      \Pi_h(e_h^u,e_h^\lambda),\ddiv\bm\Phi
    \big) = \big(
      \bm\nabla\Pi_h(e_h^u,e_h^\lambda), \bm\Phi
    \big).
  \]
  It follows  
  \begin{align*}
    {} &
    -\big(
      \Pi_h(e_h^u,e_h^\lambda),\ddiv\bm\Phi
    \big) \\
    ={} &
    \big(
      \bm\nabla\Pi_h(e_h^u,e_h^\lambda), (I-\mathcal P_h^0)\bm\Phi
    \big) +
    \big(
      \bm\nabla\Pi_h(e_h^u,e_h^\lambda),\mathcal P_h^0\bm\Phi
    \big) \\
    ={} &
    \big(
      \bm\nabla\Pi_h(e_h^u,e_h^\lambda),(I-\mathcal P_h^0)\bm\Phi
    \big) +
    \sum_{T \in \mathcal T_h} \dual{
        \Pi_h(e_h^u,e_h^\lambda), \mathcal P_T^0\bm\Phi\cdot\bm n
      }_{\partial T}
    \quad \text{(by integration by parts)} \\
    ={} &
    \big(
      \bm\nabla\Pi_h(e_h^u,e_h^\lambda),(I-\mathcal P_h^0)\bm\Phi
    \big) +
    \sum_{T \in \mathcal T_h} \dual{
      e_h^\lambda, \mathcal P_T^0\bm\Phi\cdot\bm n
    }_{\partial T}
    \quad \text{(by Lemma~\ref{lem:basic-Pi_h})}.
  \end{align*}
  Thus, we get
  \begin{align*}
    \norm{e_h^u}^2
    &=
    \big( e_h^u-\Pi_h(e_h^u,e_h^\lambda),e_h^u \big) +
    (\Pi_h(e_h^u,e_h^\lambda),e_h^u) \\
    &=
    \big(e_h^u - \Pi_h(e_h^u,e_h^\lambda), e_h^u\big) -
    \big(\Pi_h(e_h^u,e_h^\lambda),\ddiv\bm\Phi\big)
    \quad \text{(by~\eqref{eq:dual})} \\
    &=
    \mathbb I_1 + \mathbb I_2 + \mathbb I_3,
  \end{align*}
  where
  \begin{align*}
    \mathbb I_1 &:= \big(e_h^u-\Pi_h(e_h^u,e_h^\lambda),e_h^u\big), \\
    \mathbb I_2 &:= \big(
      \bm\nabla\Pi_h(e_h^u,e_h^\lambda), (I-\mathcal
      P_h^0)\bm\Phi
    \big), \\
    \mathbb I_3 &:= \sum_{T \in \mathcal T_h}
    \dual{e_h^\lambda, \mathcal P_T^0\bm\Phi\cdot\bm n}_{\partial T}.
  \end{align*}
  By~\eqref{eq:lxy4} it holds
  \[
    \mathbb I_1 \lesssim h
    \Big(
      \vertiii{(e_h^u,e_h^\lambda,\bm e_h^{\bm\sigma})}_h +
      \norm{(I-\mathcal P_h^k)\bm\sigma}
    \Big) \norm{e_h^u}.
  \]
  By~\eqref{eq:lxy3} it holds
  \[
    \mathbb I_2 \lesssim
    \Big(
      \vertiii{(e_h^u,e_h^\lambda,\bm e_h^{\bm\sigma})}_h +
      \norm{(I-\mathcal P_h^k)\bm\sigma}
    \Big) \norm{(I-\mathcal P_h^0)\bm\Phi}.
  \]

  Now let us estimate $ \mathbb I_3 $. In view of~\eqref{eq:model}, \eqref{eq:dual} and
  integration by parts, we obtain
  \begin{align*}
    {} &
    \big(\bm c(\bm\sigma - \bm\sigma_h),\bm\Phi\big)
    = (\bm\sigma - \bm\sigma_h, \bm c \bm\Phi)
    = (\bm\sigma - \bm\sigma_h, \bm\nabla\phi) \\
    ={} &
    (f,\phi) +
    \sum_{T \in \mathcal T_h} \Big(
      (\ddiv\bm\sigma_h,\phi)_T -
      \dual{\bm\sigma_h\cdot\bm n,\phi}_{\partial T}
    \Big) \\
    ={} &
    (f,\phi) +
    \sum_{T \in \mathcal T_h} \Big(
      (\ddiv\bm\sigma_h,\mathcal P_T^{k+1}\phi)_T -
      \dual{\bm\sigma_h\cdot\bm n,\mathcal P_T^\partial\phi}_{\partial T}
    \Big).
  \end{align*}
  From \eqref{eq:hdg-2} it follows
  \[
    (\ddiv_h\bm\sigma_h,\mathcal P_h^{k+1}\phi)
    = -(f,\mathcal P_h^{k+1}\phi) +
    \sum_{T \in \mathcal T_h}
    \dual{
      \alpha_T(\mathcal P_T^\partial u_h - \lambda_h),
      \mathcal P_T^{k+1}\phi
    }_{\partial T}.
  \]
  From \eqref{eq:hdg-3} it follows
  \[
    \sum_{T \in \mathcal T_h}
    \dual{\bm\sigma_h\cdot\bm n, \mathcal P_T^\partial\phi}_{\partial T}
    = \sum_{T \in \mathcal T_h} \dual{
      \alpha_T(\mathcal P_T^\partial u_h - \lambda_h),
      \mathcal P_T^\partial\phi
    }_{\partial T}
    = \sum_{T \in \mathcal T_h} \dual{
      \alpha_T(\mathcal P_T^\partial u_h - \lambda_h), \phi
    }_{\partial T}.
  \]
 The above three equations lead to
  \begin{align*}
    \big(\bm c(\bm\sigma - \bm\sigma_h), \bm\Phi\big)
    &= \big(f,(I-\mathcal P_h^{k+1})\phi\big) -
    \sum_{T \in \mathcal T_h}
    \dual{
      \alpha_T(\mathcal P_T^\partial u_h-\lambda_h),
     (I-\mathcal P_T^{k+1})\phi
    }_{\partial T} \\
    &= \big((I-\mathcal P_h^{k+1})f, (I-\mathcal P_h^{k+1})\phi \big) -
    \sum_{T \in \mathcal T_h}
    \dual{
      \alpha_T(\mathcal P_T^\partial u_h - \lambda_h),
      (I-\mathcal P_T^{k+1})\phi
    }_{\partial T}.
  \end{align*}
This equation, together with   \eqref{eq:err-1} and the fact that $ \ddiv \mathcal P_T^0\bm\Phi = 0$ for all $ T \in
  \mathcal T_h $,  gives
  \begin{align*}
    \mathbb I_3 & =
    \big(\bm c(\bm\sigma_h - \bm\sigma), \mathcal P_h^0\bm\Phi\big) \\
    &=
    \big(\bm c(\bm\sigma_h - \bm\sigma),\bm\Phi\big) -
    \big(\bm c(\bm\sigma_h - \bm\sigma), (I-\mathcal P_h^0)\bm\Phi\big) \\
    &=
    -\big((I-\mathcal P_h^{k+1})f, (I-\mathcal P_h^{k+1})\phi \big) +
    \sum_{T \in \mathcal T_h}
    \dual{
      \alpha_T(\mathcal P_T^\partial u_h - \lambda_h),
      (I-\mathcal P_T^{k+1})\phi
    }_{\partial T} -
    \big(\bm c(\bm\sigma_h - \bm\sigma), (I-\mathcal P_h^0)\bm\Phi\big).
  \end{align*}
  Using the Cauchy Schwarz inequality and the fact that $ \alpha_T = h_T^{-1}$,
  we obtain
  \begin{align}
    \mathbb I_3 \lesssim &
    \norm{(I-\mathcal P_h^{k+1})f} \norm{(I-\mathcal P_h^{k+1})\phi} +
    \norm{\bm\sigma - \bm\sigma_h} \norm{(I-\mathcal P_h^0)\bm\Phi} +
    \sum_{T \in \mathcal T_h}
      \norm{\alpha_T^\frac12(\mathcal P_T^\partial u_h - \lambda_h)}_{\partial T}
      \norm{\alpha_T^\frac12(I-\mathcal P_T^{k+1})\phi}_{\partial T} \nonumber \\
    \lesssim&
    \norm{(I-\mathcal P_h^{k+1})f} \norm{(I-\mathcal P_h^{k+1})\phi} +
    \norm{\bm\sigma - \bm\sigma_h} \norm{(I-\mathcal P_h^0)\bm\Phi} + {}
    \nonumber \\
    &
    \left(
      \sum_{T \in \mathcal T_h} h_T^{-1} \norm{
        (\mathcal P_T^\partial u_h - \lambda_h)
      }_{\partial T}^2
    \right)^\frac12
    \left(
      \sum_{T \in \mathcal T_h} h_T^{-1} \norm{
        (I-\mathcal P_T^{k+1})\phi
      }_{\partial T}^2
    \right)^\frac12. \label{eq:3452}
  \end{align}
  Using the definition~\eqref{eq:def-e_h} of $ e_h^u $ and $ e_h^\lambda $, we
  obtain
  \begin{align*}
    {} &
    \sum_{T \in \mathcal T_h} h_T^{-1} \norm{
      \mathcal P_T^\partial u_h - \lambda_h
    }_{\partial T}^2 \\
    ={} &
    \sum_{T \in \mathcal T_h} h_T^{-1} \norm{
      \mathcal P_T^\partial e_h^u - e_h^\lambda +
      \mathcal P_T^\partial(\mathcal P_T^{k+1}u -u)
    }_{\partial T}^2 \\
    \lesssim{} &
    \sum_{T \in \mathcal T_h} \Big(
      h_T^{-1}\norm{
        \mathcal P_T^\partial e_h^u - e_h^\lambda
      }_{\partial T}^2 + h_T^{-1} \norm{
        \mathcal P_T^\partial(\mathcal P_T^{k+1}u - u)
      }_{\partial T}^2
    \Big) \\
    \lesssim{} &
    \sum_{T \in \mathcal T_h} \Big(
      h_T^{-1} \norm{
        \mathcal P_T^\partial e_h^u - e_h^\lambda
      }_{\partial T}^2 + h_T^{-1} \norm{
        (I-\mathcal P_T^{k+1})u
      }_{\partial T}^2
    \Big).
  \end{align*}
  From the definition~\eqref{eq:def_vertiii_h} of $ \vertiii{\cdot}_h $, it
  follows that
  \[
    \sum_{T \in \mathcal T_h} h_T^{-1} \norm{
      \mathcal P_T^\partial u_h - \lambda_h
    }_{\partial T}^2
    \lesssim
    \vertiii{(e_h^u,e_h^\lambda,\bm e_h^{\bm\sigma})}_h^2 +
    \sum_{T \in \mathcal T_h} h_T^{-1}
    \norm{
      (I-\mathcal P_T^{k+1})u
    }_{\partial T}^2.
  \]
  Collecting~\eqref{eq:3452} and the above estimate, we obtain
  \begin{align*}
    \mathbb I_3 \lesssim
    & \norm{(I-\mathcal P_h^{k+1})f} \norm{(I-\mathcal P_h^{k+1})\phi} +
    \norm{\bm\sigma - \bm\sigma_h} \norm{(I-\mathcal P_h^0)\bm\Phi} + {} \\
    &
    \left(
      \vertiii{(e_h^u,e_h^\lambda,\bm e_h^{\bm\sigma})}_h +
      \left(
        \sum_{T \in \mathcal T_h} h_T^{-1}\norm{(I-\mathcal P_T^{k+1})u}_{\partial T}^2
      \right)^{\frac 12}
    \right)
    \left(
      \sum_{T \in \mathcal T_h} h_T^{-1}\norm{(I-\mathcal P_T^{k+1})\phi}_{\partial T}^2
    \right)^{\frac 12}.
  \end{align*}
  Using standard approximation properties of the $L^2$-orthogonal projection, we
  get, for all $ T \in \mathcal T_h $,
  \begin{align*}
    h_T^{-1}\norm{(I-\mathcal P_T^{k+1})u}_{\partial T}^2
    &\lesssim \verti{(I-\mathcal P_T^{k+1})u}_{1,T}^2,\\
    h_T^{-1}\norm{(I-\mathcal P_T^{k+1})\phi}_{\partial T}^2
    &\lesssim \verti{(I-\mathcal P_T^{k+1})\phi}_{1,T}^2.
  \end{align*}
  Thus, it follows 
  \begin{align*}
    \mathbb I_3 \lesssim
    & \norm{(I-\mathcal P_h^{k+1})f} \norm{(I-\mathcal P_h^{k+1})\phi} +
    \norm{\bm\sigma - \bm\sigma_h} \norm{(I-\mathcal P_h^0)\bm\Phi} + {} \\
    &
    \left(
      \vertiii{(e_h^u,e_h^\lambda,\bm e_h^{\bm\sigma})}_h +
      \verti{(I-\mathcal P_h^{k+1})u}_{1,h}
    \right)
    \verti{(I-\mathcal P_h^{k+1})\phi}_{1,h}.
  \end{align*}

  Finally, using these estimates for $ \mathbb I_1 $, $ \mathbb I_2 $, $ \mathbb
  I_3 $, Young's inequality, and the fact that
  \[
    \norm{\bm\sigma - \bm\sigma_h} \lesssim
    \vertiii{(e_h^u,e_h^\lambda,\bm e_h^{\bm\sigma})}_h +
    \norm{(I-\mathcal P_h^k)\bm\sigma},
  \]
  we easily obtain
  \[
    \begin{split}
      \norm{e_h^u}^2 & \lesssim
      h^2\vertiii{(e_h^u,e_h^\lambda,\bm e_h^{\bm\sigma})}_h^2 +
      h^2\norm{(I-\mathcal P_h^k)\bm\sigma}^2 +
      \norm{(I-\mathcal P_h^0)\bm\Phi}
      \Big(
        \vertiii{(e_h^u,e_h^\lambda,\bm e_h^{\bm\sigma})}_h +
        \norm{(I-\mathcal P_h^k)\bm\sigma}
      \Big) \\
      & \quad {} +
      \norm{(I-\mathcal P_h^{k+1})f} \norm{(I-\mathcal P_h^{k+1})\phi} +
      \left(
        \vertiii{(e_h^u,e_h^\lambda,\bm e_h^{\bm\sigma})}_h +
        \verti{(I-\mathcal P_h^{k+1})u}_{1,h}
      \right)
      \verti{(I-\mathcal P_h^{k+1})\phi}_{1,h},
    \end{split}
  \]
  which indicates~\eqref{eq:conv_u}, and thus completes the proof.
\end{proof}

By the above theorem, standard approximation properties of the $ L^2
$-orthogonal projection, Young's inequality, and Corollary~\ref{coro1}, we
immediately derive the following corollary.
\begin{coro}\label{coro2}
  Suppose that the dual problem~\eqref{eq:dual} satisfies the following
  regularity estimate:
  \[
    \norm{\phi}_{2,\Omega} + \norm{\bm\Phi}_{1,\Omega} \lesssim
    \norm{e_h^u}.
  \]
  Then it holds
  \begin{equation}
    \norm{u - u_h} \lesssim h \left(
      \vertiii{(e_h^u,e_h^\lambda,\bm e_h^{\bm\sigma})}_h +
      \norm{(I-\mathcal P_h^k)\bm\sigma} + h \norm{(I-\mathcal P_h^{k+1})f} +
      \verti{(I-\mathcal P_h^{k+1})u}_{1,h}
    \right) +
    \norm{(I-\mathcal P_h^{k+1})u}.
  \end{equation}
Furthermore, if $ f \in H^s(\Omega) $, $ \bm\sigma \in H^{s+1}(\Omega) $, and $ u
  \in H^{s+2}(\Omega) $ for a nonnegative integer $ s $, then 
  \begin{equation}
    \norm{u - u_h}   \lesssim h^{\min\{s+2,k+2\}}
    \left(
      \norm{f}_{s,\Omega} + \norm{\bm\sigma}_{s+1,\Omega} +
      \norm{u}_{s+2,\Omega}
    \right).
\end{equation}
\end{coro}

\section{Flux postprocessing}\label{sec_post}

In this section we follow the idea in \cite{super_con_dg, projection_based_hdg}
to construct a local postprocessing so as to obtain a new flux approximation $
\bm\sigma_h^* \in H(\ddiv;\Omega) $. We shall show that $ \bm\sigma_h^* $
converges at the same order as $ \bm\sigma_h $, while its divergence converges at
one higher order than $ \bm\sigma_h $.

Define
\begin{equation}
  \label{sigma*}
  \bm\sigma_h^* := \bm\sigma_h - \widetilde{\bm\sigma}_h,
\end{equation}
where, for any $T\in\mathcal T_h $,
\[
  \widetilde{\bm\sigma}_h|_T\in \text{RT}_{k+1}(T) :=
  \left\{
    \bm\tau: \bm\tau = \bm p + q \bm x,
    \ \bm p \in [P_{k+1}(T)]^d,\ q \in P_{k+1}(T)
  \right\}
\]
satisfies
\begin{subequations}
  \begin{align}
    (\widetilde{\bm\sigma}_h, \bm q)_T  =  0 & & &
    \text{for all $ \bm q \in [P_k(T)]^d $},\label{eq:def-post-1} \\
    \bint{\widetilde{\bm\sigma}_h\cdot\bm n, \mu}{F} =
    \bint{\alpha_T(\mathcal P_T^\partial u_h - \lambda_h), \mu}{F} & & &
    \text{for all $ \mu \in P_{k+1}(F) $ and face $F$ of $T$}.\label{eq:def-post-2}
  \end{align}
\end{subequations}
We note that the existence and uniqueness of $ \widetilde{\bm\sigma}_h $ follow
from the property of the RT elements \cite{RT}.

We have the following theorem.

\begin{thm}
  \label{thm_appro_post_sigma}
  It holds
  \begin{equation}
    \bm\sigma_h^*  \in H(\ddiv;\Omega) \quad \text{and} \quad
    \ddiv\bm\sigma_h^* =  \mathcal P_h^{k+1} \ddiv \bm\sigma.
  \end{equation}
  Moreover, it holds
  \begin{equation}\label{esti-1}
    \norm{\bm\sigma - \bm\sigma_h^*} \lesssim
    \vertiii{(e_h^u,e_h^\lambda,\bm e_h^{\bm\sigma})}_h +
    \norm{\bm\sigma - \bm\sigma_h} +
    \verti{(I-\mathcal P_h^{k+1})u}_{1,h},
  \end{equation}
  which implies
  \begin{equation} \label{esti-2}
    \norm{\bm\sigma - \bm\sigma_h^*} \lesssim h^{\min\{s+1,k+1\}}
    \left(
      \norm{f}_{s,\Omega} + \norm{\bm\sigma}_{s+1,\Omega} +
      \norm{u}_{s+2,\Omega}
    \right),
  \end{equation}
  if $ f \in H^s(\Omega) $, $ \bm\sigma \in H^{s+1}(\Omega) $, and $ u \in
  H^{s+2}(\Omega) $ for a nonnegative integer $ s $.
\end{thm}
\begin{proof}
  By \eqref{eq:hdg-3} and \eqref{eq:def-post-2}, it is easy to verify
  that
  \[
    (\bm\sigma_h^*\cdot \bm n)|_{\partial T} = 0 \quad
    \text{for all $ T \in \mathcal T_h $.}
  \]
  This implies $ \bm\sigma_h^*\in H(\ddiv;\Omega) $. From \eqref{eq:hdg-2} it
  follows,   for all $ q \in P_{k+1}(T) $,
  \[
    -(q,\ddiv\bm\sigma_h)_T + \dual{
      \alpha_T(\mathcal P_T^\partial u_h - \lambda_h), q
    }_{\partial T} = (f, q)_T.
  \]
 Then, using integration by parts, we obtain,  for all $ q \in P_{k+1}(T) $,
  \[
    (\nabla q, \bm\sigma_h)_T - \dual{
      \bm\sigma_h\cdot\bm n - \alpha_T(\mathcal P_T^\partial u_h - \lambda_h),
      q
    }_{\partial T} = (f, q)_T.
  \]
  In view of~\eqref{eq:def-post-1}--\eqref{eq:def-post-2},
  this implies
  \[
    (\nabla q, \bm\sigma_h - \widetilde{\bm\sigma}_h)_T - \bint{(\bm\sigma_h -
    \widetilde{\bm\sigma}_h)\cdot\bm n , q}{\partial T}
    = (f, q)_T\quad \text{for all $ q \in P_{k+1}(T) $,}
  \]
i.e.
  \[
    (\nabla q, \bm\sigma_h^*)_T - \dual{\bm\sigma_h^*\cdot\bm n , q}_{\partial T}
    = (f, q)_T \quad \text{for all $ q \in P_{k+1}(T) $.}
  \]
 From integration by parts it follows 
  \[
    -(q, \ddiv\bm\sigma_h^*)_T = (f, q)_T \quad
    \text{for all $ q \in P_{k+1}(T) $,}
  \]
  or equivalently,
  \[
    (\ddiv \bm\sigma_h^*)|_T = \mathcal P_T^{k+1}\ddiv\bm\sigma.
  \]

  Now let us show~\eqref{esti-1}. By \eqref{eq:def-post-1} and
  \eqref{eq:def-post-2}, a simple scaling argument yields
  \[
    \norm{\widetilde{\bm\sigma}_h}_T \lesssim h_T^{-\frac12}
    \norm{
      \mathcal P_T^\partial u_h - \lambda_h
    }_{\partial T} \quad \text{for all $ T \in \mathcal T_h $.}
  \]
  By~\eqref{eq:def-e_h} it holds
  \begin{align*}
    {} &
    \norm{\mathcal P_T^\partial u_h - \lambda_h}_{\partial T} \\
    ={} &
    \norm{
      \mathcal P_T^\partial e_h^u - e_h^\lambda +
      \mathcal P_T^\partial(\mathcal P_T^{k+1} u - u)
    }_{\partial T} \\
    \leqslant{} &
    \norm{\mathcal P_T^\partial e_h^u - e_h^\lambda}_{\partial T} +
    \norm{(I-\mathcal P_T^{k+1})u}_{\partial T}.
  \end{align*}
  Using the above two estimates, we obtain
  \begin{align*}
    \norm{\widetilde{\bm\sigma}_h} &
    \lesssim \left(
      \sum_{T\in\mathcal T_h} h_T^{-1}
      \norm{
        \mathcal P_T^\partial e_h^u - e_h^\lambda
      }^2_{\partial T}
    \right)^\frac12 +
    \left(
      \sum_{T\in\mathcal T_h} h_T^{-1}
      \norm{ (I-\mathcal P_T^{k+1})u }_{\partial T}^2
    \right)^\frac12.
  \end{align*}
  Since standard approximation properties of the $ L^2 $-orthogonal
  projection yield
  \[
    h_T^{-1} \norm{(I-\mathcal P_T^{k+1})u}_{\partial T}^2 \lesssim
    \verti{(I-\mathcal P_T^{k+1})u}_{1,T}^2 \quad
    \text{for all $ T \in \mathcal T_h $,}
  \]
  it follows 
  \begin{equation}
    \label{eq:lxy-0129}
    \norm{\widetilde{\bm\sigma}_h}
    \lesssim \left(
      \sum_{T\in\mathcal T_h} h_T^{-1}
      \norm{
        \mathcal P_T^\partial e_h^u - e_h^\lambda
      }^2_{\partial T}
    \right)^\frac12 + \verti{(I-\mathcal P_h^{k+1})u}_{1,h}.
  \end{equation}
 Finally, \eqref{esti-1} follows from  \eqref{sigma*},
  \eqref{eq:lxy-0129}, and the definition~\eqref{eq:def_vertiii_h} of $
  \vertiii{\cdot}_h $, and    \eqref{esti-2} follows from
  \eqref{esti-1}, Theorem~\ref{thm:conv_sigma}, and Corollary~\ref{coro1}. This
  completes the proof.
\end{proof}

\section{Numerical experiments}\label{sec_num}
This section provides numerical experiments in two-space dimensions to verify our theoretical results.
We consider the problem \eqref{eq:model} with $ \Omega = (0,1) \times (0, 1) $ and
\begin{displaymath}
  {\bm c}(x,y) = \left (
  \begin{array}{cc}
    1 + x^2y^2 & 0\\
    0 & 1+x^2y^2\\
  \end{array}
\right),
\end{displaymath}
and we set $u(x,y) = \sin(\pi x) \sin(\pi y) $ to be the analytic solution. \par

We start with an initial mesh shown in Figure \ref{initial mesh} with $h^{-1} = 2$
and obtain a sequence of refined meshes by bisection. Numerical results are presented
in Tables \ref{table_numerical}-\ref{table_new_flux}  for the HDG method
\eqref{discretization} with $k = 0, 1 $.

Table \ref{table_numerical} shows the history of convergence for the potential
approximation $u_h $ and the flux approximation $ \bm\sigma_h $. We can see that
for $k=0$, which corresponds to the lowest order HDG method, the potential
error $ \norm{u-u_h}$ is of second-order accuracy, and the flux error
$ \norm{\bm\sigma - \bm\sigma_h}$ is of first-order accuracy, while for
$k=1 $, $ \norm{u-u_h}$ is of third-order accuracy and $ \norm{\bm\sigma - \bm\sigma_h}$
is of second-order accuracy. These numerical results are conformable
to the error estimates in Theorems \ref{thm:conv_sigma}-\ref{thm:conv_u}  and Corollaries \ref{coro1}-\ref{coro2}.

Table \ref{table_new_flux}  shows the history of convergence for the
postprocessed flux approximation $ \bm\sigma_h^*$. We can see that for
$k=0$, $ \norm{\bm\sigma-\bm\sigma_h^*} $ is of first-order accuracy,
and $ \norm{\ddiv\bm\sigma-\ddiv\bm\sigma_h^*}$ is of second-order accuracy,
while for $k=1 $, $ \norm{\bm\sigma-\bm\sigma_h^*} $ is of second-order
accuracy and $ \norm{\ddiv\bm\sigma-\ddiv\bm\sigma_h^*}$ is of third-order
accuracy. These numerical results are conformable to the error estimates
in Theorem \ref{thm_appro_post_sigma}.

\begin{figure}[H]
  \begin{center}
  \includegraphics[width=0.4\linewidth]{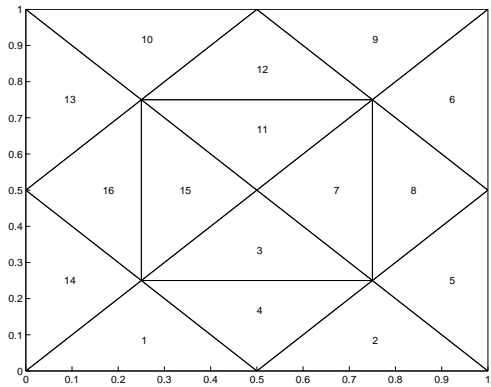}
  \end{center}
  \caption{Initial mesh with $h^{-1} = 2$}\label{initial mesh}
\end{figure}

\begin{table}[H]
  \begin{center}
  \begin{tabular}{p{2.cm} p{2.cm} p{2.cm} p{2.cm} p{2.cm} p{2.cm}}
    \toprule
    Degree $k$ & Mesh $h^{-1}$ & \multicolumn{2}{l}{$ \norm{u-u_h}$} &
    \multicolumn{2}{l}{$ \norm{\bm\sigma - \bm\sigma_h}$} \\
    \cmidrule(r){3-4} \cmidrule(l){5-6}
    &               & {Error}          &  {Order}        & {Error}        & {Order} \\
    \midrule
    0 &      2 & 3.052e-1 & -               & 1.230 & -\\
    &      4 & 7.828e-2 & 1.963 & 6.443e-1 & 0.933\\
    &      8 & 1.968e-2 & 1.992 & 3.250e-1 & 0.987\\
    &      16 & 4.927e-3 & 1.998 & 1.629e-1 & 0.997\\
    &      32 & 1.232e-3 & 1.999 & 8.147e-2 & 0.999\\
    1 &      2 & 3.431e-2 & -               & 2.524e-1 & -\\
    &      4 & 4.376e-3 &  2.971 & 6.211e-2 & 2.023\\
    &      8 & 5.510e-4 & 2.990 & 1.552e-2 & 2.000\\
    &      16 & 6.900e-5 &  2.997 & 3.882e-3 & 2.000\\
    \bottomrule
  \end{tabular}
  \caption{History of convergence for $u_h $ and $ \bm\sigma_h $}\label{table_numerical}
  \end{center}
\end{table}

\begin{table}[H]
  \begin{center}
  \begin{tabular}{p{2.cm} p{2.cm} p{2.cm} p{2.cm} p{2.cm} p{2.cm}}
    \toprule
    Degree $k$ & Mesh $h^{-1}$ & \multicolumn{2}{l}{$ \norm{\bm\sigma-\bm\sigma_h^*}$} &
    \multicolumn{2}{l}{$ \norm{\ddiv\bm\sigma-\ddiv\bm\sigma_h^*}$} \\
    \cmidrule(r){3-4} \cmidrule(r){5-6}
    &               & {Error}          &  {Order}          & {Error}  & {Order}\\
    \midrule
    0 &      2 & 1.080 & -               &  0.7003 & -\\
    &      4 & 5.616e-1 & 0.944 &  0.1861 & 1.912 \\
    &      8 & 2.826e-1 & 0.991 &  0.0470 & 1.985\\
    &      16 & 1.415e-1 & 0.998 &  0.0118 & 1.996\\
    &      32 & 7.078e-2 & 0.999 &  0.0029 & 1.999\\
    1 &      2 & 2.278e-1 & -               &  0.0114 &-\\
    &      4 & 5.514e-2 &  2.046 &  0.0014 & 3.015 \\
    &      8 & 1.373e-2 & 2.006 &  1.7919e-4 & 2.997\\
    &      16 & 3.429e-3 & 2.001 &  2.2405e-5 & 2.999\\
    \bottomrule
  \end{tabular}
  \caption{History of convergence for $ \bm\sigma_h^*$}\label{table_new_flux}
  \end{center}
\end{table}

\end{document}